\documentclass{amsproc}
\usepackage[all]{xy}
\usepackage{amscd}
\usepackage{amssymb}
\newtheorem{theorem}{Theorem}[section]
\newtheorem{lemma}[theorem]{Lemma}
\newtheorem{corollary}[theorem]{Corollary}
\newtheorem{proposition}[theorem]{Proposition}

\theoremstyle{definition}
\newtheorem{definition}[theorem]{Definition}

\theoremstyle{remark}
\newtheorem{remark}[theorem]{Remark}

\numberwithin{equation}{section}

 \def\={\setminus}  
 
 \def\ind{\mathop\mathrm{ind}\nolimits}
\def\md#1{#1\mbox{-}\mathrm{mod}}
\def\vec{\mathsf{vec}}
\def\bp{\bullet}    \def\iso{\simeq}

\def\add{\mathop\mathrm{add}\nolimits}
\def\Fun{\mathop\mathrm{Fun}\nolimits}
\def\Rep{\mathop\mathrm{Rep}\nolimits}
\def\rep{\mathop\mathrm{rep}\nolimits}
\def\supp{\mathop\mathrm{supp}\nolimits}

\def\gdim{\mathop\mathrm{gl.dim}}

\def\iff{if and only if }

\def\setsuch#1#2{\left\{\,#1\,|\,#2\,\right\}}

\def\rad{\mathop\mathrm{rad}}

\def\ob{\mathop\mathrm{ob}\nolimits}

\def\Ker{\mathop\mathrm{Ker}}
\def\im{\mathop\mathrm{Im}}
\def\8{\infty}      \def\+{\oplus}
\def\*{\otimes} \def\dd{\partial}

\def\op{\mathop\mathrm{op}}

\def\al{\alpha} \def\be{\beta}  
\def\Ga{\Gamma}       \def\la{\lambda}
     
 \def\Si{\Sigma} \def\eps{\varepsilon}

\def\fR{\mathbf r}

\def\mN{\mathbb N}  \def\mZ{\mathbb Z}
\def\mA{\mathbb A}  \def\mD{\mathbb D}
\def\mE{\mathbb E}

\def\kA{\mathcal A}     
\def\kC{\mathcal C} \def\kD{\mathcal D}
\def\kK{\mathcal K} \def\kV{\mathcal V}
\def\kW{\mathcal W} \def\kQ{\mathcal Q}
\def\kP{\mathcal P}     
 
 \def\kI{\mathcal I}
\def\kX{\mathcal X}

\def\dA{\mathfrak A}    
    \def\dP{\mathfrak P}

      \def\sJ{\mathsf J}

\def\sR{\mathsf R}

     \def\oV{\overline{\mathcal V}}

\def\Mk{\Bbbk}
\def\oV{\overline{\kV}}

\begin{document}

\title{Derived categories for algebras with radical square zero}

\author{Viktor  Bekkert}
\address{Departamento de Matem\'atica, ICEx, Universidade Federal de Minas
Gerais, Av.  Ant\^onio Carlos, 6627, CP 702, CEP 30123-970, Belo
Horizonte-MG, Brasil}

\email{bekkert@mat.ufmg.br}
\thanks{The first author was supported by FAPESP (Grant N 98/14538-0) and CNPq (Grant 301183/00-7).}

\author{Yuriy  Drozd}
\address{Institute of Mathematics, National Academy of Sciences of Ukraine, Te\-reschen\-kivska 3, 01601 Kiev, Ukraine}
\email{drozd@imath.kiev.ua}


\dedicatory{Dedicated to Ivan Shestakov for his 60th
birthday.}


\begin{abstract}
{We determine the derived representation types of algebras with radical square zero and give a description of the
indecomposable objects in their bounded derived categories. }
\end{abstract}

\maketitle

\section*{Introduction}

This paper is based on a talk given by the first author
at the Conference "Algebras, Representations and Applications" (Sao
Paulo, Brazil, August 2007).

Let $\kA$ be a finite dimensional algebra over an algebraically
closed field $\Mk$, $\md\kA$ be the category of left finitely
generated $\kA$-modules and let $\kD^b(\kA)$ be the bounded derived
category of the category $\md\kA$.

The category  $\kD^b(\kA)$ is known for few algebras $\kA$. For
example, the structure of $\kD^{b}(\kA)$ is well-known for
hereditary algebras of finite and tame type \cite{H} and for
tubular algebras \cite{HR}.

In the present paper we investigate the derived category
$\kD^b(\kA)$ for the finite dimensional algebras with radical
square zero.

The structure of the paper is as follows. In Section~\ref{s1}
preliminary results about derived categories are given. We replace
finite dimensional algebras by \emph{locally finite dimensional
categories} (shortly  \emph{lofd}). If such a category only has
finitely many indecomposable objects, this language is equivalent to
that of finite dimensional algebras.

In Section~\ref{s2} for a given lofd category $\kA$ with radical
square zero, we construct following \cite{BD} a box such that its representations
classify the objects of the derived category $\kD^b(\kA)$,
 which is used in the next sections.

It follows from \cite{BD} that every lofd category over an
algebraically closed field is either derived tame or derived wild.
In Section~\ref{s3} we establish the derived representation type for
lofd categories with radical square zero.

A description of indecomposables in $\kD^{b}(\kA)$ is  given in
Section~\ref{s4}. Namely, we reduce this problem to the problem of
description of indecomposables finite dimensional modules for some
hereditary path algebra $\Mk\kQ_\kA$. In derived tame cases we
describe indecomposables in $\kD^{b}(\kA)$ explicitly.

After this paper was finished, we were told that similar results had been obtained by R. Bautista and S. Liu \cite{BL}.
Note that they use quite different methods.

\section{Derived categories}\label{s1}

 We will follow in general the notations and terminology of \cite{BD} (see also \cite{D3}, \cite{D4}).

 We consider categories and algebras over a fixed algebraically closed field $\Mk$. A $\Mk$-category $\kA$ is called
 \emph{locally finite dimensional} (shortly \emph{lofd}) if the following conditions hold:
 \begin{enumerate}
\item  All spaces $\kA(x,y)$ are finite dimensional for all
objects $x,y$.
 \item
  $\kA$ is \emph{fully additive}, i.e. it is additive and all idempotents in it split.\\
  Conditions 1,2 imply that the category $\kA$ is \emph{Krull--Schmidt},  i.e. each object uniquely decomposes into
 a direct sum of indecomposable objects; moreover, it is \emph{local}, i.e. for each indecomposable object $x$ the algebra
 $\kA(x,x)$ is local. We denote by $\ind\kA$ a set of representatives of isomorphism classes of indecomposable objects
 from $\kA$.
 \item
  For each object $x$ the set $\setsuch{y\in\ind\kA}{\kA(x,y)\ne0\text{ or }\kA(y,x)\ne0}$ is finite.
\end{enumerate}

 We denote by $\vec$ the category of finite dimensional vector spaces over $\Mk$ and by $\md\kA$ the category of
 \emph{finite dimensional $\kA$-modules}, i.e. functors $M:\kA\to\vec$ such that $\setsuch{x\in\ind\kA}{Mx\ne0}$
 is finite.

 For an arbitrary category $\kC$ we denote by $\add\kC$ the minimal fully additive category containing $\kC$. For instance,
 one can consider $\add\kC$ as the category of finitely generated projective $\kC$-modules; especially, $\add\Mk=\vec$.
 We denote by $\Rep(\kA,\kC)$ the category of functors $\Fun(\kA,\add\kC)$ and call them \emph{representations} of
 the category $\kA$ in the category $\kC$. Obviously, $\Rep(\kA,\kC)\iso\Rep(\add\kA,\kC)$. If the category $\kA$ is lofd,
 we denote by $\rep(\kA,\kC)$ the full subcategory of $\Rep(\kA,\kC)$ consisting of the representations $M$ with \emph{finite
 support} $\supp M=\setsuch{x\in\ind\kA}{Mx\ne0}$. In particular, $\rep(\kA,\Mk)=\md\kA$.

We recall that a quiver is locally finite if at most finitely many
arrows start or stop at each vertex. We recall also that every lofd category is equivalent to a quiver category, i.e. $\kA=\add\Mk \kQ/\kI$,
where $\kQ=\kQ_\kA$ is the locally finite quiver of $\kA$ and $\kI$
is an admissible ideal in the path category $\Mk \kQ$ of $\kQ$.

 We denote by $\kD(\kA)$ (respectively, $\kD^b(\kA)\,$) the \emph{derived category} (respectively,
(two-sided) bounded derived category) of the category $\kA$-mod,
where $\kA$ is a lofd category.
 These categories are triangulated categories. We denote the shift
 functor by $[1]$, and its inverse by $[-1]$.
 Recall that $\kA^{\op}$ embeds as a full subcategory into $\md\kA$. Namely, each object $x$ corresponds to the functor
 $\kA^x=\kA(x,\_\,)$. These functors are projective in the category $\md\kA$; if $\kA$ is fully additive, these are
 all projectives (up to isomorphism). On the other hand,
 $\md\kA$ embeds as a full subcategory into $\kD^b(\kA)$: a module $M$ is treated as the complex only having a
 unique nonzero component equal $M$ at the $0$-th position. It is also known that $\kD^b(\kA)$ can be identified
 with the category $\kK^{-,b}(\kA)$ whose objects are right bounded complexes of projective modules with bounded homology
 (that is, complexes of finitely generated projective modules with the property that the homology groups are non zero only at
 a finite number of places) and morphisms
 are homomorphisms of complexes modulo homotopy \cite{GM}. If $\gdim\kA<\8$, every bounded complex has
 a bounded projective resolution, hence $\kD^b(\kA)$ can identified with $\kK^b(\kA)$, the category of
 bounded projective  complexes modulo homotopy, but that is not the case if $\gdim\kA=\8$.
 Moreover, if $\kA$ is lofd, we can confine the considered complexes by \emph{minimal} ones, i.e.
 always suppose that $\im d_n\subseteq\rad P_{n-1}$ for all $n$.
 We denote by $\kP^b_{\min}(\kA)$ the category
 of minimal bounded complexes of  projective $\kA$-modules.

Given  $M \in \kD^b(\kA)$, we denote by $ P_{M} $  the minimal projective
resolution of $M$.

 For $P\not=0 \in \kK^{b}(\kA)$,
let $t$ be the minimal  number such that $P_{i} = 0$ for $i > t$.
Then, $\be(P)$ denotes the {\em (good) truncation} of $P$ below
$t$, i.e. the complex given by
$$
\be(P)_{i} =\left\{ \begin{array}{lll}
P_{i} & \mbox{, if $i\leq t$;}\\
\Ker d(P)_{t} & \mbox{, if $i=t+1$;}\\
0 & \mbox{, otherwise,}
\end{array}
\right.
$$

$$
d(\be(P))_i =\left\{ \begin{array}{lll}
d(P)_i & \mbox{, if $i\leq t$;}\\
i_{\Ker d(P)_{t}} & \mbox{, if $i=t+1$;}\\
0 & \mbox{, otherwise,}
\end{array}
\right.
$$

\noindent where $i_{\Ker d^{t}_{P}} $ is the obvious embedding.

Let $\overline{{\kX}\,(\kA)}=$ $\{\,M\in \ind\kP^b_{\min}(\kA)\,|$
$P_{\beta(M)}\not\in \kK^{b}(\kA)\,\}$. Let $\cong_{\kX}$ be the
equivalence relation on the set $\overline{{\kX}\,(\kA)}$ defined
by $M\cong_{\kX} N$ \iff $P_{\beta(M)}\cong$ $P_{\beta(N)}$ in
$\kK^{-,b}(\kA)$. We use the notation $\kX\,(\kA)$ for a fixed set
of representatives of the quotient set $\overline{{\kX}\,(\kA)}$
over the equivalence relation $\cong_{\kX}$.

\begin{proposition}\cite{BM} \label{ver}
 $\ind\kD^b(\kA) = \ind\kP^b_{\min}(\kA) \cup
\{\be(M)\,|\, M \in {\kX}\,(\kA)\}.$
\end{proposition}

\begin{remark}
If $\kA$ has finite global dimension, we have
$\kX\,(\kA)=\emptyset$ and hence $\ind \kD^b(\kA) =
\ind\kP^b_{\min}(\kA)$.
\end{remark}

We recall the definitions of derived tame and derived wild lofd
categories from \cite{BD}.

  \begin{definition}\label{tw}
  \begin{enumerate}
\item   The \emph{rank} of an object $x\in\kA$ (or of the
corresponding projective module $\kA^x$) is the function
 $\fR(x):\ind\kA\to\mZ$ such that $x\iso\bigoplus_{y\in\ind\kA}\fR(x)(y)y$. The \emph{vector rank} $\fR_\bp(P_\bp)$
 of a bounded complex of projective $\kA$-modules is the sequence $(\dots,\fR(P_n),\fR(P_{n-1}),\dots)$ (actually it
 has only finitely many nonzero entries).
 \item
  We call a \emph{rational family} of bounded minimal complexes over $\kA$ a bounded complex $(P_\bp,d_\bp)$
 of finitely generated projective $\kA\*\sR$-modules, where $\sR$ is a \emph{rational algebra},
 i.e. $\sR=\Mk[t,f(t)^{-1}]$ for a nonzero polynomial $f(t)$, and $\im d_n\subseteq\sJ P_{n-1}$
 For such a complex   we define $P_\bp(m,\la)$, where $m\in\mN,\,\la\in\Mk,\,f(\la)\ne0$, the complex
 $(P_\bp\*_\sR\sR/(t-\la)^m,d_\bp\*1)$. It is indeed a complex of projective $\kA$-modules. We put $\fR_\bp(P_\bp)=
 \fR_\bp(P_\bp(1,\la))$ (this vector rank does not depend on $\la$).
 \item
  We call a lofd category $\kA$ \emph{derived tame} if there is a set $\dP$ of rational families of bounded complexes over
 $\kA$ such that:
  \begin{enumerate}
 \item  For each vector rank $\fR_\bp$ the set $\dP(\fR_\bp)=\setsuch{P_\bp\in\dP}{\fR_\bp(P_\bp)=\fR}$ is finite.
 \item
  For each vector rank $\fR_\bp$ all indecomposable complexes $(P_\bp,d_\bp)$ of projective $\kA$-modules
 of this vector rank, except finitely many isomorphism classes, are isomorphic to $P_\bp(m,\la)$ for some $P_\bp\in\dP$
 and some $m,\la$.
\end{enumerate}
 The set $\dP$ is called a \emph{parameterising set} of $\kA$-complexes.
 \item
  We call a lofd category $\kA$ \emph{derived wild} if there is a bounded complex $P_\bp$ of projective modules over
$\kA\*\Si$, where $\Si$ is the free $\Mk$-algebra in 2 variables,
such that, for every finite dimensional $\Si$-modules
 $L,L'$,
  \begin{enumerate}
\item   $P_\bp\*_\Si L\iso P_\bp\*_\Si L'$ \iff $L\iso L'$.
 \item
  $P_\bp\*_\Si L$ is indecomposable \iff so is $L$.
\end{enumerate}
 (It is well-known that then an analogous complex of $\kA\*\Ga$-modules exists for every finitely generated $\Mk$-algebra
 $\Ga$.)
\end{enumerate}
 \end{definition}

 Note that, according to these definitions, every \emph{derived discrete} (in particular, \emph{derived finite}) lofd category
 \cite{V} is derived tame (with the empty set  $\dP$).

It was proved in \cite{BD} that every lofd category over an
algebraically closed field is either derived tame or derived wild.

\section{Related boxes}\label{s2}

Recall (see \cite{D1}, \cite{D2}) that a {\em box} is a pair $\dA
= (\kA, \kV)$ consisting of a category $\kA$ and an $\kA$-coalgebra
$\kV$. We denote by $\mu$ the comultiplication in $\kV$, by $\eps$
its counit and by $\oV=\ker\eps$ its {\em kernel}. We always suppose
that $\dA$ is {\em normal}, i.e. there is a {\em section} $\omega :
x \to \omega_{x}\,\, (x\in \ob \kA)$ such that $\eps(\omega_{x}) =
1_{x}$ and $\mu(\omega_{x}) = \omega_{x}\otimes \omega_{x}$ for all
$x$. A category $\kA$ is called {\em free} if it is isomorphic to a
path category $\Mk\kQ $ of a quiver $\kQ$. A normal box $\dA = (\kA,
\kV)$ is called {\em free} if so is the category $\kA$, while the
kernel $\oV$ is a free $\kA$-bimodule.

Recall that the {\em differential} of a normal box $\dA = (\kA,
\kV)$ is the pair $\dd = (\dd_{0}, \dd_{1})$ of mappings, $\dd_{0} :
\kA \to \oV$, $\dd_{1} : \oV \to \oV\otimes_{\kA}\oV$, namely
$$ \dd_{0}a = a\omega_{x}- \omega_{y}a\,\, \textrm{for}\,\, a\in \kA(x, y),$$
$$ \dd_{1}v = \mu(v) - v\otimes \omega_{x} -\omega_{y}\otimes v\,\, \textrm{for}\,\,
v\in \oV(x, y).$$

A {\em representation} of a box $\dA = (\kA, \kV)$ over a category
$\kC$ is defined as a functor $M : \kA \to \add \kC$. A {\em
morphism} of such representations $f : M\to N$ is defined as a
homomorphisms of $\kA$-modules $\kV\otimes_{\kA}M\to N$. If $g : N
\to L$ is another morphism, there product is defined as the
composition

$$ \kV\otimes_\kA M
\stackrel{\mu\otimes 1}\rightarrow \kV\otimes_{\kA}\kV\otimes_{\kA}
M \stackrel{1\otimes f}\rightarrow \kV\otimes_{\kA} N
        \stackrel{g}\rightarrow L\,.  $$

 Thus we obtain the {\em category of representations} $\Rep(\dA, \kC)$. If $\dA$ is a
free box, we denote by $\rep(\dA, \kC)$ the full subcategory of
$\Rep(\dA, \kC)$ consisting of representations with finite support
$\supp M = \{ x\in \ob\kA |Mx \not= 0 \}$. If $\kC = \vec$, we write
$\Rep(\kA)$ and $\rep(\kA)$.

Given a lofd $\kA$ with radical square zero, we are going to
construct a box such that its representations classify the objects
of the derived category $\kD^b(\kA)$ (see \cite{BD}, \cite{D3} for the case
of an arbitrary lofd category).

Let $\kQ=\kQ_\kA$ be a quiver of $\kA$. Given two vertices $a$ and
$b$ we define $\kQ_1[a,b]$ as the set of all arrows from $a$ to
$b$.
Given an arrow $a$ of $\kQ$, let us denote by $a^{-1}$ a formal
inverse of $a$, and let us set $s(a^{-1})=t(a)$ and $t(a^{-1})=s(a)$.
By a \emph{walk} $w$ of length $n$ we mean a sequence $w_1w_2\cdots w_n$
where each $w_i$ is either of the form $a$ or $a^{-1}$, $a$ being an arrow in $\kQ$ and
where $s(w_{i+1})=t(w_i)$ for $1\leq i<n$. For each walk $w=w_1w_2\cdots w_n$ we define
$s(w)=s(w_1)$ and $t(w)=t(w_n)$. By definition, a \emph{closed walk} is a walk $w$ such that
$s(w)=t(w)$.

 Consider the path category $\kA^{\Box}=\Mk \kQ^{\Box}$, where
 $\kQ^{\Box}$ is the
quiver with the set of points $\kQ^{\Box}_{0}=\kQ_0\times \mZ$ and
with the set of arrows $\kQ^{\Box}_{1}=\kQ_1\times\mZ$, where for
given $\al:a\to b$ in $\kQ$ we set $s((\al,i))=(b,i)$ and
$t((\al,i)):=(a,i-1)$.

Consider the normal free box $\dA=\dA(\kA)=(\kA^{\Box},\kW)$, with the
kernel $\overline{\kW}$ freely generated by the set
$\{\varphi_{\al,i}\, |\, \al\in \kQ_1, i\in\mZ\}$, where
$s(\varphi_{\al,i})=(t(\al),i)$, $t(\varphi_{\al,i})=(s(\al),i)$
and with zero differential $\dd$.

Given a box $\dA$, we denote by $\rep(\dA)$ the category of finite
dimensional representations of $\dA$.

Let us consider the following functor ${\bf F}:
\rep(\dA(\kA))\rightarrow \kP^b_{\min}(\kA)$.

 A representation $M\in\rep(\dA)$ is given by vector spaces $M(x,n)$ and linear mappings
 $M(\al,n): M(y,n)\to M(x,n-1)$, where $\al\in\kQ_1[x,y]$
 and $x,y\in\kQ_0,\,n\in\mZ$.
 For such a representation, set $P_n=\bigoplus_{x\in \kQ_0} \kA^x\*M(x,n)$
 and $d_n=\bigoplus_{x,y\in\kQ_0}\sum_{\al\in\kQ[x,y]}\kA^\al\*M(\al,n)$. A morphism
$\Psi:M\to M^\prime$ is given by linear mappings
$\Psi(z,n):M(z,n)\to M^\prime(z,n)$ and
$\Psi(\varphi_{\al,m}):M(y,n)\to M^\prime(x,n)$, where
$x,y,z\in\kQ_0$ and $\al\in\kQ_1[x,y]$. We define a homomorphism
$F(\Psi):F(M)\to F(M^\prime)$ by the following rule. Given
$x,y\in\kQ_0$ we set
$$
\widetilde{\kQ}_1[x,y] =\left\{
\begin{array}{ll}
{\kQ}_1[x,y] &\mbox{, if $x\ne y$}\\
{\kQ}_1[x,y]\cup\{1_x\} & \mbox{, otherwise.}
\end{array}
\right.
$$

For given $\al\in\widetilde{\kQ}_1[x,y]$ we set
$$
\Psi_{\al,n} =\left\{
\begin{array}{ll}
\Psi(\varphi_{\al,n}) &\mbox{, if $\al\in\kQ_1$}\\
\Psi(x,n) & \mbox{, otherwise.}
\end{array}
\right.
$$

Then $F(\Psi)$ is defined by
$F(\Psi)_n=\bigoplus_{x,y\in\kQ_0}\sum_{\al\in\widetilde{\kQ}[x,y]}\kA^\al\*\Psi_{\al,n}$.

 The following
Theorem follows from the Theorem 2.2 in \cite{BD}.

\begin{theorem}\label{equiv}
${\bf F}$ is an equivalence of categories.
\end{theorem}

We define a shift functor in $\rep(\dA(\kA))$ by the following
rule. Given $M\in \rep(\dA(\kA))$ and $j\in\mZ$ we define $M[j]\in
\rep(\dA(\kA))$ by $M[j](a,i)=M(a,i-j)$ and
$M[j](\al,i)=M(\al,i-j)$. In the same way we can define $[j]$ for
morphisms in $\rep(\dA(\kA))$.

\begin{lemma}\label{shift}
${\bf F}(M[i])={\bf F}(M)[i]$ and ${\bf F}(\varphi[i])={\bf
F}(\varphi)[i]$ for any object $M$ and any morphism $\varphi$ in
$\rep(\dA(\kA))$.
\end{lemma}
\begin{proof}
Straightforward.
\end{proof}

Given a quiver $\kQ$ we fix some vertex $a\in\kQ_0$ and denote by
$\kQ[i]$ the connected component of $\kQ^{\Box}$  which contain
the vertex $(a,i)$. Given a walk $w=w_1\cdots w_n$ in $\kQ$ we
denote by $\eps^+(w)$  (resp.,  $\eps^-(w)$ ) the number of $w_i$ of the form $p$ (resp., $p^{-1}$), $p$ being
an arrow. We set $\eps(w)=|\eps^+(w)-\eps^-(w)|$. We denote by
$\kQ^{c}$ the set of all closed walks in $\kQ$ and set
$\eps(\kQ)=\min_{w\in\kQ^{c}}\eps(w)$ in case of
$\kQ^{c}\ne\emptyset$ and $\eps(\kQ)=0$ otherwise. We say that a quiver $\kQ$ satisfies the \emph{walk condition}
provided $\eps(\kQ)=0$ (= the number of clockwise oriented arrows is the same as
the number of counterclockwise oriented arrows for any
closed walk $w$ of $\kQ$).

\begin{lemma}\label{path}
Let $b\in \kQ_0$, $i,j\in\mZ$ and $a$ as above. Then
$(b,j)\in\kQ[i]$ \iff there exists a walk $w$ from
$a$ to $b$ in $\kQ$ such that $j=i+\eps^+(w)-\eps^-(w)$.
\end{lemma}
\begin{proof}
Straightforward.
\end{proof}

\begin{corollary}\label{component}
Let $\kQ$ be a connected quiver. Then $\kQ[i]=\kQ[j]$ \iff there
exists a closed walk $w$ in $\kQ$ such that $i\equiv j\,\,
(\mod \eps(w))$.
\end{corollary}
\begin{proof}
Let $a\in \kQ$ be as above.

 $"\Longrightarrow."$
Suppose that $\kQ[i]=\kQ[j]$ for some $i\ne j\in\mZ$. Then
$(a,j)\in \kQ[i]$ and by Lemma~\ref{path} there exists a
walk $w$ from $a$ to $a$ in $\kQ$ such that
$j=i+\eps^+(w)-\eps^-(w)$, hence $i\equiv j\,\, (\mod \eps(w))$.

$"\Longleftarrow."$ Let $w$ be a closed walk in $Q$ such
that $i\equiv j$\,\, $(\mod \eps(w))$. Then $j=i+m\eps(w)$ for
some $m\in\mZ$. Since $\kQ$ is connected, there exists a
walk $u$ from $a$ to $s(w)$. Then for the
closed walk $v=u^{-1}w^mu$ in $\kQ$ we have
$j=i+\eps^+(v)-\eps^-(v)$. Therefore $(a,j)\in\kQ[i]$ by
Lemma~\ref{path} and hence $\kQ[j]=\kQ[i]$.
\end{proof}

\begin{corollary}\label{component1}
Let $\kQ$ be a connected quiver. Then $\kQ[i]=\kQ[j]$ \iff
$i\equiv j\,\, (\mod \eps(\kQ))$.
\end{corollary}
\begin{proof}
It is easy to see that if $\kQ$ is connected and
$\kQ^{c}\ne\emptyset$, then for any closed walk $w$ we have
$\eps(w)=m\eps(\kQ)$ for some $m\in\mN$. Hence the statement
follows from Corollary~\ref{component}.
\end{proof}

\begin{lemma}\label{cycle}
\begin{enumerate}
\item Let $\kQ$ be a connected quiver which satisfies the walk
condition. Then $\kQ^{\Box}$ is a disjoint union
$\bigsqcup_{i\in\mZ}\kQ[i]$, where $\kQ[i]\simeq \kQ^{\op}$ for all $i$.

\item Let $\kQ$ be a quiver which not satisfy the walk
condition. Then $\kQ^{\Box}$ is the disjoint union
$\bigsqcup_{0\leq i<\eps(\kQ)}\kQ[i]$, where
$\kQ[i]\simeq \kQ^{\diamond}$ for all $i$ for some quiver
$\kQ^{\diamond}$.
\end{enumerate}
\end{lemma}

\begin{proof}
\begin{enumerate}
\item It follows from Corollary~\ref{component} that if $i\ne j$
we have $\kQ[i]\ne\kQ[j]$. It is easy to see that in this case
$\kQ[i]\iso \kQ^{\op}$.

\item It is easy to see that $\kQ[i]\iso\kQ[j]$ for all
$i,j\in\mZ$. Therefore the statement follows from
Corollary~\ref{component}.
\end{enumerate}
\end{proof}

\section{Derived representation type}\label{s3}

\begin{theorem}\label{drt}

Let $\kA$ be a lofd connected category with radical square zero.

\begin{enumerate}
\item $\kA$ is derived tame \iff  $\kQ_\kA$ is a Dynkin quiver (of
types $\mA_n\, (n\geq 1)$, $\mD_n\, (n\geq 4)$, $\mE_n\, (8\geq
n\geq 6)$) or an Euclidian quiver (of types $\tilde {\mA}_n\,
(n\geq 1)$, $\tilde {\mD}_n\, (n\geq 4)$, $\tilde {\mE}_n\, (8\geq
n\geq 6)$) or a quiver of types $\mA_\8$, $\mA_\infty^{\8}$ or
$\mD_\8$.

\item $\kA$ is derived discrete \iff
 $\kQ_\kA$ is a Dynkin quiver
or an Euclidian quiver $\tilde {\mA}_n\, (n\geq 1)$ which does not
satisfy the walk condition or a quiver of types $\mA_\8$,
$\mA_\8^{\8}$ or $\mD_\8$.

\item $\kA$ is derived finite \iff $\kQ_\kA$ is a Dynkin quiver.
\end{enumerate}
\end{theorem}

\begin{proof}
Let $\kQ=\kQ_\kA$. We distinguish three cases.

\medskip
\noindent(a) $\kQ$ has no cycles.

Then by Lemma~\ref{cycle} we have in this case
$\kQ^{\Box}=\bigsqcup_{i\in\mZ}\kQ[i]$, where $\kQ[i]=\kQ^{\op}$.
Hence the statements of the Theorem in this case follow from
\cite{G}, \cite{N} and Proposition~\ref{ver}.

\medskip
\noindent(b) $\kQ$ is an Euclidian quiver $\tilde {\mA}_n$.

It follows from Lemma~\ref{cycle} that if $\kQ$ satisfies the
walk condition, then $\kQ^{\Box}=\bigsqcup_{i\in \mZ}\kQ[i]$,
where $\kQ[i]=\kQ^{\op}$, hence $\kA$ is derived tame, but is not
derived discrete by \cite{N} and Proposition~\ref{ver}; and if
$\kQ$ does not satisfy the walk condition, then by Lemma~\ref{cycle}
we have $\kQ^{\Box}=\bigsqcup_{0\leq i < \eps(\kQ)}\kQ[i]$, where
$\kQ[i]\iso\mA^\8_\8$ for all $i$, hence $\kA$ is derived discrete
by \cite{G} and Proposition~\ref{ver}.

\medskip
\noindent(c) $\kQ$ has an Euclidian sub-quiver $\kQ^{\prime}\ne
\kQ$ of type $\tilde {\mA}_n$.

It follows from (b) that $\kQ^{\prime\Box}$ has connected
sub-quiver $X$ of type $\tilde {\mA}_n$ or
$\mA^{\infty}_{\infty}$. Let $X^{\prime}$ be the connected
sub-quiver of $\kQ^{\Box}$ which contains $X$. Since
$\kQ^{\prime}\ne \kQ$, we have $X^{\prime}\ne X$. Therefore  $\Mk
\kQ^{\Box}$ is wild by \cite{N} and hence $\kA$ is derived wild.
\end{proof}

\section{Indecomposable objects}\label{s4}

Let ${\bf F}$ be as in Section~\ref{s2} and $\kX$ as in
Section~\ref{s1}.

\begin{theorem}\label{ind}
Let $\kA$ be a lofd category with radical square zero. Then the
complexes ${\bf F}(M)$ and $\beta(N)$, where $M\in \ind \rep(\Mk
\kQ^{\Box})$ and $N\in\kX(\Mk \kQ^{\Box})$,  constitute an
exhaustive list of pairwise non-isomorphic indecomposable objects
of ${\kD}^{b}(\kA).$
\end{theorem}

\begin{proof}
Since $\dd=0$ for the box $\dA=\dA(\kA)$, we have
 that $\ind\rep(\dA)=\ind\rep(\Mk \kQ^{\Box})$.
Hence the statement follows from Theorem~\ref{equiv} and Proposition~\ref{ver}.
\end{proof}

For a quiver $\kQ$ which satisfies the walk condition, we
denote by $\imath: \rep(\kQ^{\op})\to \rep(\kQ^{\Box})$ the
inclusion functor which sent $\kQ^{\op}$ to $\kQ[0]$. It follows
from Lemma~\ref{cycle} that in this case $\gdim\kA<\8$ ( because
 the quiver $\kQ$ has no oriented cycle) and it is the disjoint
union $\bigsqcup_{i\in\mZ}\kQ[i]$, where $\kQ[i]=\kQ^{\op}$ for
all $i$. Hence we obtain the following Corollary.

\begin{corollary}\label{cind}
Let $\kA$ be a lofd category with radical square zero whose quiver
$\kQ=\kQ_\kA$ satisfies the walk condition. Then the complexes
${\bf F}(\imath(M))[i]$, where $M\in \ind \rep(\Mk \kQ^{\op})$ and
$i\in\mZ$, constitute an exhaustive list of pairwise
non-isomorphic indecomposable objects of ${\kD}^{b}(\kA).$
\end{corollary}

\bibliographystyle{amsalpha}

\end{document}